\documentclass[11pt,a4paper]{article}

\usepackage{amsmath,amssymb,latexsym}
\usepackage{mathtools}
\usepackage{graphicx}
\usepackage{epsfig}
\usepackage{longtable}
\usepackage{amsthm}
\usepackage{xcolor}

\usepackage{lineno}
\newcommand{\R}{\mbox{${\rm{I\!R}}$}}

\newtheorem{theorem}{Theorem}
\newtheorem{lemma}[theorem]{Lemma}

\newtheorem{definition}{Definition}
\newtheorem{example}{Example}
\newtheorem{remark}{Remark}

\def\XX{{\mathbf X}}

\def\xx{{\mathbf x}}

\def\uu{{\mathbf u}}
\def\UU{{\mathbf U}}

\begin{document}

\title{ \Large\bf Modeling directional monotonicity with copulas}
\author{Enrique de Amo$^{\rm a}$, David García-Fernández$^{\rm b}$, José Juan Quesada-Molina$^{\rm c}$,\\ Manuel Úbeda-Flores$^{\rm a}$\footnote{Corresponding author}\\
\small{$^{\rm a}$Department of Mathematics, University of Almería, 04120 Almería, Spain}\\
\small{\texttt{edeamo@ual.es,\,\,\,mubeda@ual.es}}\\
\small{$^{\rm b}$Research Group of Theory of Copulas and Applications, University of Almería, 04120 Almería, Spain}\\
\small{\texttt{davidgfret@correo.ugr.es}}\\
\small{$^{\rm c}$Department of Applied Mathematics, University of Granada, 18071 Granada, Spain}\\
\small{\texttt{jquesada@ugr.es}}
}
\maketitle

\begin{abstract}
    The purpose of this paper is to characterize the concept of monotonicity according to a direction related to a set of $n$ random variables in terms of its associated $n$-copula C. We start establishing relationships in the bivariate and trivariate cases, which help to understand the extension to the multivariate case. Examples of copulas in all the studied cases are provided.
\end{abstract}

\section{Introduction}\label{sec:copulas}
$\phantom{999}$Dependence among random variables is a widely field of research in statistics and probability.  Analyzing this dependence structure is crucial when we want to figure out the behaviour of a complex model components. Therefore, a deep research of dependence relations could bring us comprehensively information about the model of interest.

Random variables can be related in several ways, presenting different relationships of dependence. One of the most important in literature is positive dependence, which can be characterize as the inclination of components within a random vector to assume concordant values. Negative dependence can be defined similarly, with the exception that now, random variables values moves in different directions. Both of these concepts may not be enough to condense all the  dependence that the variables show. It is for that reason that on this paper, we will focus on the concept of monotonicity according to a direction $\alpha\in\mathbb{R}^n$ defined in \cite{Que2024}, whose positive (respectively, negative) dependence concept is denoted by I$(\alpha)$ (respectively, D$(\alpha)$).

Copulas ---multivariate distribution functions with univariate uniform marginals--- serve as a valuable tool for examining the positive dependence characteristics of a random vector. This is because they encapsulate the dependence structure of the corresponding multivariate distribution function, regardless of the individual marginal distributions \cite{Ne06}. Additionally, they provide scale-free measures of dependence and serve as a foundation for constructing families of distributions \cite{Fisher97}.

In this paper, as we have already mentioned, our objective is to explore multivariate copulas associated to random vectors that exhibit the I($\alpha$) property---we will refer to these as I($\alpha$) copulas. However, for simplicity---and to better understand the general case---, we will start with the bivariate and trivariate cases and then we will extrapolate our results to the multivariate case.

This paper is organized as follows: we begin with some preliminary concepts and results corresponding to multivariate dependence, specifically monotonicity according to a direction, and to copula theory (Section \ref{sec:pre}). Straightaway, in Section \ref{sec:main}, we will characterize the concept of I$(\alpha)$ for the bivariate case using copulas, providing several examples of copulas that have this kind of dependence in different directions $\alpha\in\R^2$. Comparably, we will carry out a similar study in the trivariate case. Once we get used to this concepts in lower dimensions, we will extrapolate the characterization using copulas to the multivariate case, formulating a general result, allowing us to obtain an inequality, for every direction $\alpha\in\R^n$, in terms of the $n$-copula associated with the random $n$-vector and its marginals copulas. Finally, Section \ref{sec:4} is devoted to discuss the conclusions derived from our research.

\section{Preliminaries}\label{sec:pre}

$\phantom{999}$Let $n$ be a natural number such that $n\ge 2$, and let $\XX=(X_1,X_2,\ldots,X_n)$ be an $n$-dimensional random vector. In the following, the expression ``nonincreasing in $\mathbf{x}$''--—and similarly for nondecreasing---means that it is nonincreasing in each of the components of $\mathbf{x}=(x_1,x_2,\ldots,x_n)$, and $\mathbf{X} \leq \mathbf{x}$ means $X_i \leq x_i$ for all $i=1,2,\ldots,n$:

Regarding multivariate dependence, the following two positive dependence notions, introduced in \cite{Ha70}, are widely known:
\begin{itemize}
    \item [i)] $\XX$ is \emph{left corner set decreasing}, denoted by LCSD$(\XX)$, if 
    $$\mathbb{P}[\XX\leq \xx| \XX\leq \xx']\text{ is nonincreasing in } \xx'\text{ for all }\xx.$$
    \item [ii)] $\XX$ is \emph{right corner set increasing}, denoted by RCSI$(\XX)$, if 
    $$\mathbb{P}[\XX> \xx| \XX> \xx']\text{ is nondecreasing in } \xx'\text{ for all }\xx.$$
\end{itemize}
The equivalent negative dependence concepts LCSI$(\XX)$ \emph{(left corner set increasing)} and RCSD$(\XX)$ \emph{(right corner set decreasing)} are defined exchanging ``nondecreasing" and ``nonincreasing" in their respective expressions.

These concepts of multivariate positive dependence can be extended to the concept of monotonicity according to a direction, defined bellow, which allow us to capture new dependency structures among random variables.
\begin{definition}[\cite{Que2024}]
    Let $\XX$ be a $n$-dimensional random vector and $\alpha=(\alpha_1,\alpha_2,...,\alpha_n)\in\R^n$ such that $|\alpha_i|=1$ for all $i=1,2,...,n$. The random vector $\XX$, or its joint distribution function, is said to be increasing (respectively, decreasing) according to the direction $\alpha$, denoted by $I(\alpha)$ (respectively, $D(\alpha))$, if 
    $$\mathbb{P}[\alpha \XX>{\bf x}|\alpha \XX> {\bf x'}]$$
    is nondecreasing (respectively, nonincreasing) in ${\bf x'}$ for all ${\bf x}$.
\end{definition}

In this paper we will focus on the I$(\alpha)$ dependence concept. Similar results can be obtain for D$(\alpha)$. Note that the concept I$(\alpha)$ generalizes the RCSI and LCSD concepts mentioned above; that is, LCSD corresponds to I$({\bf -1})$ and RCSI corresponds to I$({\bf 1})$, where ${\bf 1}=(1,1,\ldots,1)$.

%Let $H$ be the joint distribution function of $\XX$. The $I(-%\bf 1)$ condition can be written as 
%$$\frac{H(x\wedge x')}{H(x')} \text{ is nonincreasing in } x' %\text{ for all } x,$$
%where $x\wedge x'=(\min\{x_1,x_1'\},\min\{x_2,x_2'),...,\min\
%{x_n,x_n'\})$. If we denote $\overline{H}$ the survival %function of $H$; i.e., $\overline{H}(x)=\mathbb{P}[X>x]$, %then the $I(\bf 1)$ condition can be written as
%$$\frac{H(x\vee x')}{H(x')} \text{ is nondecreasing in } x' %\text{ for all } x,$$
%where $x\vee x'=(\max\{x_1,x_1'\}, \max\{x_2,x_2'\},...,\max\
%{x_n,x_n'\}).$

Now we recall some notions related to copulas. For $n\ge 2$, an $n$-dimensional {\it copula} ($n$-copula, for short) is the restriction to $[0,1]^n$ of a continuous {\it n}-dimensional distribution function whose univariate marginals are uniform on $[0,1]$. The importance of copulas in statistics is described in the following result due to Abe Sklar \cite{Sk59}: Let ${\bf X}=(X_1,X_2,\ldots,X_n)$ be a random vector with joint distribution function $F$ and one-dimensional marginal distributions $F_{1},F_2,\ldots,F_{n}$, respectively. Then there exists an $n$-copula $C$, which is uniquely determined on $\times_{i=1}^{n}$Range$F_{i}$, such that
$$F({\bf x})=C(F_{1}(x_{1}),F_2(x_2),\ldots,F_{n}(x_{n}))\quad\rm{ for}\,\,\rm{ all}\,\,{\bf x}\in [-\infty,+\infty]^{n}$$
(for a complete proof of this result, see \cite{Ub2017}). Thus, copulas link joint distribution functions to their one-dimensional marginals. For a survey on copulas, see \cite{Durante2016book,Ne06}, and for some results about positive dependence properties by using copulas can be seen, for instance, in \cite{Joe1997,Mu06,Navarro2021,Ne06,Wei2014}.

Let $\Pi^n$ denote the $n$-copula for independent random variables (or product $n$-copula), i.e., $\Pi^n({\bf u})=\prod_{i=1}^{n}u_i$ for all ${\bf u}=(u_1,u_2,\ldots,u_n)\in[0,1]^n$.

For any {\it n}-copula $C$ we have
$$W^{n}({\bf u})=\max\left\{0,\sum_{i=1}^{n}u_{i}-n+1\right\}\le C({\bf u})\le \min\{u_{1},u_{2},\ldots,u_{n}\}=M^{n}({\bf u})$$
for all ${\bf u}$ in $[0,1]^{n}$. $M^{n}$ is an {\it n}-copula for all $n\ge 2$; however, $W^{n}$ is an {\it n}-copula only when $n=2$.

The definition of $d$-marginal of an $n$-copula $C$, where $1\leq d<n$, which will be useful in the study of dependence in higher dimensions, is as follows. Let $\XX$ be an $n$-dimensional random vector with associated $n$-copula $C$ and let $\sigma=(i_n,\ldots,i_m)\subseteq\{1,\ldots,n\}$ such that $1\leq m\leq n-1$. The $\sigma$-marginal copula of $C$, $C_\sigma:[0,1]^m\longrightarrow [0,1]$, is defined by setting $n-m$ arguments of $C$ equal to $1$, i.e.,
$$C_{\sigma}(u_1,\ldots,u_m)=C(v_1,\ldots,v_n),$$
where $v_i=u_i$ if $i\in\{i_1,\ldots,i_m\}$, and $v_i=1$ otherwise.

Finally, given a random vector $(X_1,X_2,\ldots,X_n)$ with $n$-copula $C$, the {\it survival} $n$-{\it copula associated with} $C$, which we denote by $\widehat{C}$, is given by
$$\widehat{C}({\bf u})=\mathbb{P}\left[X_1\ge 1-u_1,X_2\ge 1-u_2,\ldots,X_n\ge 1-u_n\right]$$
for all ${\bf u}\in[0,1]^n$.

The next section is devoted to the study of I$(\alpha)$ $n$-copulas.

\section{Monotonicity according to a direction and copulas}
\label{sec:main}

$\phantom{999}$In this section, we will characterize the I$(\alpha)$ concept for the bivariate case by using copulas, providing various examples that exhibit this type of dependence in different directions. Similarly, we will conduct an analogous analysis in the trivariate case. Once we become familiar with these concepts in lower dimensions, we will extend the description using $n$-copulas to the multivariate case, establishing a general result that enables us to derive a characterization for any direction $\alpha \in \mathbb{R}^n$ in terms of the n-copula associated with the random $n$-vector and its marginal $d$-copulas.

\subsection{The bivariate case}

$\phantom{999}$We start our study on the I$(\alpha)$ dependence concept for copulas with the bivariate case.

\begin{theorem}\label{th:bi}Let $(U,V)$ be a random pair with associated 2-copula $C$. Then $C$ is:
\begin{itemize}
\item[i.] I$(-1,-1)$ if, and only if, $C(u,v)C(u',v')\ge C(u,v')C(u',v)$ for all $u,v,u',v'$ in $[0,1]$ such that $u\le u'$ and $v\le v'$;
\item [ii.] $I(1,-1)$ if, and only if, $[v-C(u,v)][v'-C(u',v')]\leq [v-C(u',v)][v'-C(u,v')]$ for all $u,v,u',v'$ in $[0,1]$ such that $u\leq u'$ and $v\leq v'$; 
\item [iii.] $I(-1,1)$ if, and only if, $[u-C(u,v)][u'-C(u'.v')]\leq [u'-C(u',v)][u-C(u,v')]$ for all $u,v,u',v'$ in $[0,1]$ such that $u\leq u'$ and $v\leq v'$;  
\item[iv.] I$(1,1)$ if, and only if, $\widehat{C}(u,v)\widehat{C}(u',v')\ge\widehat{C}(u,v')\widehat{C}(u',v)$ for all $u,v,u',v'$ in $[0,1]$ such that $u\le u'$ and $v\le v'$.
\end{itemize}
\end{theorem}

\begin{proof}
The proof of parts i. and iv. can be found in \cite[Theorem 5.2.15 and Corollary 5.2.17]{Ne06}. Now we prove part ii.---the proof of part iii. is similar, and we omit it.

Let $(U,V)$ be a pair of random variables with associated 2-copula $C$. Assume $C$ is I$(1,-1)$, then we have that $\mathbb{P}[U>u,-V>w|U>u',-V>w']$ is nondecreasing in $(u',w')$ for all $(u,w)$, i.e.,
\begin{equation}\label{eq:partii}
\mathbb{P}[U>u,-V>w|U>u',-V>w']\le \mathbb{P}[U>u,-V>w|U>u'',-V>w'']
\end{equation}
for all $u,u',u''\in[0,1]$ and $w,w',w''\in[-1,0]$ such that $u'\le u''$ and $w'\le w''$. Equation \eqref{eq:partii} is equivalent to
$$\frac{\mathbb{P}[U>u,-V>w,U>u',-V>w']}{\mathbb{P}[U>u',-V>w']}\le \frac{\mathbb{P}[U>u,-V>w,U>u'',-V>w'']}{\mathbb{P}[U>u'',-V>w'']},$$
i.e.,
$$\frac{\mathbb{P}[U>u,V<-w,U>u',V<-w']}{\mathbb{P}[U>u',V<-w']}\le \frac{\mathbb{P}[U>u,V<-w,U>u'',V<-w'']}{\mathbb{P}[U>u'',V<-w'']};$$
therefore,
$$\frac{\mathbb{P}[U>\max\{u,u'\},V<\min\{-w,-w'\}]}{\mathbb{P}[U>u',V<-w']}\le \frac{\mathbb{P}[U>\max\{u,u''\},V<\min\{-w,-w''\}]}{\mathbb{P}[U>u'',V<-w'']},$$
that is,
$$\frac{\min\{-w,-w'\}-C(\max\{u,u'\},\min\{-w,-w'\})}{-w'-C(u',-w')}\le \frac{\min\{-w,-w''\}-C(\max\{u,u''\},\min\{-w,-w''\})}{-w''-C(u'',-w'')}.$$
Taking $v=-w$, $v'=-w'$ and $v''=-w''$, we obtain 
\begin{eqnarray*}\nonumber
[v''-C(u'',v'')]\cdot[\min\{v,v'\}\!\!\!&-&\!\!\!C(\max\{u,u'\},\min\{v,v'\})]\\ \label{eq:caracteriz}
\!\!\!&\le&\!\!\![v'-C(u',v')]\cdot[\min\{v,v''\}-C(\max\{u,u''\},\min\{v,v''\})]
\end{eqnarray*}
for all $u,v,u',v',u'',v''$ in $[0,1]$ such that $u'\le u''$ and $v''\le v'$. By setting $u=u'\leq u''$, $v'=v''$ and renaming $u'=u''$, it follows for every $v\leq v'$ and $u\leq u'$ in $[0,1]$ that
$$[v-C(u,v)][v'-C(u',v')]\leq [v-C(u',v)][v'-C(u,v')].$$

Conversely, if $[v-C(u,v)][v'-C(u'.v')]\leq [v-C(u',v)][v'-C(u,v')]$ for all $u,v,u',v'\in[0,1]$ such that $u\leq u'$ and $v\leq v'$, then,
given $u,u'\in [0,1]$ and $w,w'\in[-1,0]$, in order to prove that the random variables are I$(1,-1)$, we consider three possible cases:
\begin{enumerate}
\item If $u>u', w>w'$ we have that
$$\mathbb{P}[U>u,-V>w|U>u',-V>w']=\frac{\mathbb{P}[U>u,-V>w]}{\mathbb{P}[U>u',-V>w']}$$
is nondecreasing in $u',w'$.
\item If $u\leq u'$ and $w\leq w'$ we have that $\mathbb{P}[U>u,-V>w|U>u',-V>w']=1$, and therefore it is nondecresing in $u',w'$.
\item Let us consider, without loss of generality, $u\leq u'$ and $w>w'$. Then we have
$$\mathbb{P}[U>u,-V>w|U>u',-V>w']=\frac{\mathbb{P}[U>u',-V>w]}{\mathbb{P}[U>u',-V>w']}.$$
Since $\mathbb{P}[U>u',-V>w']\geq\mathbb{P}[U>u'',-V>w'']$ for all $u',u'',w',w''$ such that $u'\leq u''$ and $w'\leq w''$ we have that the equation above is nondecreasing in $w'$. So, in order to prove that this equation is nondecreasing in $u'$ considering $u''$ such that $u'\leq u''$, we need to verify that 
$$P[-V>w|U>u',-V>w']\leq \mathbb{P}[-V>w|U>u'',-V>w'],$$
which, taking $v=-w$ and $v'=-w'$, is equivalent to 
$$P[V<v|U>u',V<v']\leq \mathbb{P}[V<v|U>u'',V<v'],$$
Due to statement's inequality we have that 
$$\frac{v-C(u',v)}{v'-C(u',v')}\leq\frac{v-C(u'',v)}{v'-C(u'',v')}$$
for $u'\leq u''$ and $v\leq v'$. From this inequality we obtain, in terms of probability, that
$$\frac{\mathbb{P}[U>u',V<v]}{\mathbb{P}[U>u',V<v']}\leq\frac{\mathbb{P}[U>u'',V<v]}{\mathbb{P}[U>u'',V<v']}$$
\end{enumerate}
In all the cases we proved that $C$ is I$(1,-1)$, so the proof is complete.
\end{proof}

We provide several examples of $2$-copulas where we apply Theorem \ref{th:bi}.

\begin{example}The 2-copula $M^2$ is I($\alpha$) for $\alpha=(1,1)$ and $\alpha=(-1,-1)$. On the other hand, the 2-copula $W^2$ is I$(\alpha)$ for $\alpha=(-1,1)$ and $\alpha=(1,-1)$.
\end{example}

\begin{example}Let $\{C_{\delta}\}_{\delta\in[-1,1]}$ be the {\it Ali-Mikhail-Haq} one-parameter family of 2-copulas given by
$$C_{\delta}(u,v)=\frac{uv}{1+\delta(1-u)(1-v)}$$
for all $(u,v)\in[0,1]^2$ (see \cite{Ali1978}). It is easy to prove that $C_\delta$ is I$(-1,1)$ and I$(1,-1)$ for $\delta\in[0,1]$, and $C_\delta$ is I$(1,1)$ and I$(-1,-1)$ for $\delta\in[-1,0]$.
\end{example}

\begin{example}\label{ex:FGM}Let $\{C_{\lambda}\}_{\lambda \in [-1,1]}$ be the {\it Farlie-Gumbel-Morgenstern} (FGM) one-parameter family of 2-copulas given by
$$C_{\lambda}(u,v)=uv[1+\lambda(1-u)(1-v)]$$
for all $(u,v)\in[0,1]^2$ (see \cite{Ne06}). Then we have that $C_\lambda$ is I$(1,1)$ and I$(-1,-1)$ for $\lambda\in[0,1]$, and $C_\lambda$ is I$(-1,1)$ and I$(1,-1)$ for $\lambda\in[-1,0]$.
\end{example}

%Before heightening the dimension of the random vector we urge to point out that the resultant copula of a convex combination of copulas that are $I(\alpha)$ is not always $I(\alpha)$. As a counterexample we show the following copula.

%The following example shows that the converse of Proposition \ref{pro:relation} does not hold in general.

%\begin{example}\label{ex:RodUb}Consider the 2-copula $$C(u,v)=uv\left[1+(1-u)(1-v)^2\right]$$ for all $(u,v)\in[0,1]^2$. $C$ belongs to the family of 2-copulas studied in %\cite{RodUb2004}. Then we have that $C$ is IS$(1,1)$. Moreover, it is easy to prove that $C$ is not I$(1,1)$.
%\end{example}

\subsection{The trivariate case}

$\phantom{999}$In this subsection we study the I($\alpha$) dependence concept in terms of 3-copulas. Despite the difficulty of working with three random variables and the complexity of the expressions obtained, drawn results take us to similar conclusions to the bivariate case.

\begin{theorem}\label{th:three}
    Let $(U,V,W)$ be three random variables with associated $3$-copula $C$. Then $C$ is:
    \begin{itemize}
        \item [i.] $I(1,1,1)$ if, and only if, $\widehat{C}(u,v,w)\widehat{C}(u',v',w')\geq \widehat{C}(u,v',w')\widehat{C}(u',v,w)$ for all $u,v, w, u',v', w'$ in $[0,1]$ such that $u\le u'$, $v\le v'$ and $w\leq w'$.
        \item [ii.] $I(1,1,-1)$ if, and only if, $[w-C_{23}(v,w)-C_{13}(u,w)+C(u,v,w)]\cdot[w'-C_{23}(v',w')-C_{13}(u',w')+C(u',v',w')]\leq [w-C_{23}(v',w)-C_{13}(u',w)+C(u',v',w)]\cdot[w'-C_{23}(v,w')-C_{13}(u,w')+C(u,v,w')]$ for all $u,v,w,u',v',w'$ in $[0,1]$ such that $u\leq u'$, $v\leq v'$ and $w\leq w'$;
        \item[iii.] $I(1,-1,1)$ if, and only if, $[v-C_{23}(v,w)-C_{12}(u,v)+C(u,v,w)]\cdot[v'-C_{23}(v',w')-C_{12}(u',v')+C(u',v',w')]\leq [v-C_{23}(v,w')-C_{12}(u',v)+C(u',v,w')]\cdot[v'-C_{23}(v',w)-C_{12}(u,v')+C(u,v',w)]$ for all $u,v,w,u',v',w'$ in $[0,1]$ such that $u\leq u'$, $v\leq v'$ and $w\leq w'$;
        \item[iv.] $I(-1,1,1)$ if, and only if, $[u-C_{13}(u,w)-C_{12}(u,v)+C(u,v,w)]\cdot[u'-C_{13}(u',w')-C_{12}(u',v')+C(u',v',w')]\leq [u-C_{13}(u,w')-C_{12}(u,v')+C(u,v',w')]\cdot[u'-C_{13}(u',w)-C_{12}(u',v)+C(u',v,w)]$ for all $u,v,w,u',v',w'$ in $[0,1]$ such that $u\leq u'$, $v\leq v'$ and $w\leq w'$;
        \item [v.] $I(1,-1,-1)$ if, and only if, $[C_{23}(v,w)-C(u,v,w)]\cdot[C_{23}(v',w')-C(u',v',w')]\leq [C_{23}(v,w)-C(u',v,w)]\cdot[C_{23}(v',w')-C(u,v',w')]$  for all $u,v,w,u',v',w'$ in $[0,1]$ such that $u\leq u'$, $v\leq v'$ and $w\leq w'$;
        \item [vi.] $I(-1,1,-1)$ if, and only if, $[C_{13}(u,w)-C(u,v,w)]\cdot[C_{13}(u',w')-C(u',v',w')]\leq [C_{13}(u,w)-C(u,v',w)]\cdot[C_{13}(u',w')-C(u',v,w')]$ for all $u,v,w,u',v',w'$ in $[0,1]$ such that $u\leq u'$, $v\leq v'$ and $w\leq w'$;
        \item [vii.] $I(-1,-1,1)$ if, and only if, $[C_{12}(u,v)-C(u,v,w)]\cdot[C_{12}(u',v')-C(u',v',w')]\leq [C_{12}(u,v,w')-C(u,v,w')]\cdot[C_{12}(u',v')-C(u',v',w)]$ for all $u,v,w,u',v',w'$ in $[0,1]$ such that $u\leq u'$, $v\leq v'$ and $w\leq w'$;
        \item[viii.] $I(-1,-1,-1)$ if, and only if, $C(u,v,w)C(u',v',w')\geq C(u,v',w')C(u',v,w)$ for all $u,v, w, u',v', w'$ in $[0,1]$ such that $u\le u'$, $v\le v'$ and $w\leq w'$.
    \end{itemize}
\end{theorem}

\begin{proof}
We prove part ii. The rest of the parts can be proved in a similar way, so we omit their proofs.

Let $(U,V,W)$ be three random variables with associated 3-copula $C$. Assume $C$ is I$(1,1,-1)$, then we have that $$\mathbb{P}[U>u,V>v, -W>t|U>u',V>v', -W>t']$$ is nondecreasing in $(u',v',t')$ for all $(u,v,t)$ in $[0,1]^3$, i.e.,
\begin{equation}\label{eq:tricase2}
\mathbb{P}[U>u,V>v, -W>t|U>u',V>v', -W>t']\le \mathbb{P}[U>u,V>v, -W>t|U>u'',V>v'', -W>t'']
\end{equation}
for all $u,v,u',v',u'',v''\in[0,1]$ and $t,t',t''\in[-1,0]$ such that $u'\le u''$, $v'\le v''$ and $t'\leq t''$.  Equation \eqref{eq:tricase2} is equivalent to
$$\frac{\mathbb{P}[U>u,V>v,-W>t,U>u',V>v',-W>t']}{\mathbb{P}[U>u',V>v',-W>t']}\le \frac{\mathbb{P}[U>u,V>v,-W>t,U>u'',V>v'',-W>t'']}{\mathbb{P}[U>u'',V>v'',-W>t'']},$$
i.e.,
$$\frac{\mathbb{P}[U>u,V>v,W<-t,U>u',V>v',W<-t']}{\mathbb{P}[U>u',V>v',W<-t']}\le \frac{\mathbb{P}[U>u,V>v,W<-t,U>u'',V>v'',W<-t'']}{\mathbb{P}[U>u'',V>v'',W<-t'']};$$
hence,
\begin{align*}
&\frac{\mathbb{P}[U>\max\{u,u'\},V>\max\{v,v'\}],W<\min\{-t,-t'\}}{\mathbb{P}[U>u',V>v',W<-t']}\\
\leq\hspace{0.1cm}&\frac{\mathbb{P}[U>\max\{u,u''\},V>\max\{v,v''\},W<\min\{-t,-t''\}]}{\mathbb{P}[U>u'',V>v'',W<-t'']},
\end{align*}
that is,
\begin{align*}
&\frac{\min\{-t,-t'\}-C_{23}(\max\{v,v'\},\min\{-t,-t'\})-C_{13}(\max\{u,u'\},\min\{-t,-t'\})}{-t'-C_{23}(v',-t')-C_{13}(u',-t')+C(u',v',-t')}\\
&+\frac{C(\max\{u,u'\},\max\{v,v'\},\min\{-t,-t'\})}{-t'-C_{23}(v',-t')-C_{13}(u',-t')+C(u',v',-t')}\\
&\leq\frac{\min\{-t,-t''\}-C_{23}(\max\{v,v''\},\min\{-t,-t''\})-C_{13}(\max\{u,u''\},\min\{-t,-t''\})}{-t''-C_{23}(v'',-t'')-C_{13}(u'',-t'')+C(u'',v'',-t'')}\\
&+\frac{C(\max\{u,u''\},\max\{v,v''\},\min\{-t,-t''\})}{-t''-C_{23}(v'',-t'')-C_{13}(u'',-t'')+C(u'',v'',-t'')}.
\end{align*}
Taking $w=-t$, $w'=-t'$ and $w''=-t''$ we obtain
\begin{eqnarray}\nonumber
[\min\{w,w'\}\!\!\!&-&\!\!\!C_{13}(\max\{u,u'\},\min\{w,w'\})-C_{23}(\max\{v,v'\},\min\{w,w'\})\\ \nonumber
            &+&\!\!\!C(\max\{u,u'\},\max\{v,v'\},\min\{w,w'\})]\cdot[w''-C_{13}(u'',w'')-C_{23}(v'',w'')+C(u'',v'',w'')]\\ \nonumber
            &\leq&\!\!\! [\min\{w,w''\}-C_{13}(\max\{u,u''\},\min\{w,w''\})-C_{23}(\max\{v,v''\},\min\{w,w''\})\\ 
            &+&\!\!\!C(\max\{u,u''\},\max\{v,v''\},\min\{w,w''\})]\cdot[w'-C_{13}(u',w')-C_{23}(v',w')+C(u',v',w')]\label{eq:tricharac}
\end{eqnarray}
for all $u,v,w,u',v', w',u'',v'', w''$ in $[0,1]$ such that $u'\le u''$, $v'\le v''$ and $w''\leq w'$. By setting $u=u'\leq u''$, $v=v'\leq v''$ and $w'=w''$ in Equation \eqref{eq:tricharac}, for every $w\leq w'$ in $[0,1]$ it follows (renaming $u'=u''$ and $v'=v''$)
\begin{align}\nonumber
    &[w-C_{23}(v,w)-C_{13}(u,w)+C(u,v,w)]\cdot[w'-C_{23}(v',w')-C_{13}(u',w')+C(u',v',w')]\\
    \leq\hspace{0.1cm} &[w-C_{23}(v',w)-C_{13}(u',w)+C(u',v',w)]\cdot[w'-C_{23}(v,w')-C_{13}(u,w')+C(u,v,w')].\label{eq:tricharac2}
\end{align}

Conversely, if Equation \eqref{eq:tricharac2} holds, then, given $u,u',v,v'\in [0,1]$ and $t,t'\in[-1,0]$, we consider three possible cases:
\begin{enumerate}
    \item If $u>u', v>v'$ and $ t>t'$ we have that
    $$\mathbb{P}[U>u,V>v,-W>t|U>u',V>v',-W>t']=\frac{\mathbb{P}[U>u,V>v,-W>t]}{\mathbb{P}[U>u',V>v',-W>t']}$$
    is nondecreasing in $u',v',t'$.
    \item If $u\leq u', v\leq v'$ and $t\leq t'$ we have that $\mathbb{P}[U>u,V>v,-W>t|U>u',V>v',-W>t']=1$, and therefore it is nondecresing in $u',v',t'$.
    \item Let us consider, without loss of generality, $u\leq u',v\leq v'$ and $t>t'$. Then we have
    \begin{align*}
        &\mathbb{P}[U>u,V>v,-W>t|U>u',V>v',-W>t']\\
        =\hspace{0.1cm}&\frac{\mathbb{P}[U>u,V>v,-W>t,U>u',V>v',-W>t']}{\mathbb{P}[U>u',V>v',-W>t']}\\
        =\hspace{0.1cm}&\frac{\mathbb{P}[U>u',V>v',-W>t]}{\mathbb{P}[U>u',V>v',-W>t']}.
    \end{align*}
    Since $\mathbb{P}[U>u',V>v',-W>t']\geq\mathbb{P}[U>v'',V>v'',-W>t'']$ for all $u',u'',v',v'',t',t''$ such that $u'\leq u'',v'\leq v'', t\leq t''$ we have that the expression above is nondecreasing in $t'$. In order to prove that this expression is nondecreasing in $u'$ and $v'$ considering $u''$ and $v''$ such that $u'\leq u''$ and $v'\leq v''$ we need to verify that 
    $$P[-W>t|U>u',V>v',-W>t']\leq \mathbb{P}[-W>t|U>u'',V>v'',-W>t'],$$
    which is equivalent to 
    $$P[W<w|U>u',V>v',W<w']\leq \mathbb{P}[W<w|U>u'',V>v'',W<w'],$$
    taking $w=-t$ and $w'=-t'$. Due to the statement of the inequality, we have 
    $$\frac{w-C_{23}(v',w)-C_{13}(u',w)+C(u',v',w)}{w'-C_{23}(v',w')-C_{13}(u',w')+C(u',v',w')}\leq\frac{w-C_{23}(v'',w)-C_{13}(u'',w)+C(u'',v'',w)}{w'-C_{23}(v'',w')-C_{13}(u'',w')+C(u'',v'',w')}$$
    for $u'\leq u'', v'\leq v''$ and $w\leq w'$. From this inequality we obtain, in terms of probability, that
    $$\frac{\mathbb{P}[U>u',V>v',W<w]}{\mathbb{P}[U>u',V>v',W<w']}\leq\frac{\mathbb{P}[U>u'',V>v'',W<w]}{\mathbb{P}[U>u'',V>v'',W<w']},$$
    which is equal to the desire expression.
 \end{enumerate}
Hence, we obtain that $C$ is $I(1,1,-1)$ in all the cases, so that the proof is complete.
\end{proof}

As an application of Theorem \ref{th:three}, we provide some examples.

\begin{example}
    The $3$-copula $M^3$ is $I(\alpha)$ for $\alpha=(1,1,1)$ and $\alpha=(-1,-1,-1)$.
\end{example}

\begin{example}\label{ex:FGM:three}Let $\{C_{\lambda}\}_{\lambda \in [-1,1]}$ be a generalization of the FGM one-parameter family of 2-copulas of Example \ref{ex:FGM}, and which is given by
$$C_{\lambda}(u_1,u_2,u_3)=u_1u_2u_3[1+\lambda(1-u_1)(1-u_2)(1-u_3)]$$
for all $u=(u_1,u_2,u_3)\in[0,1]^3$ (see \cite{Ne06}). Note that, in this case, the $2$-dimensional margins are given by $C_{i,j}(u_i,u_j)=u_iu_j$ for $1\leq i<j\leq 3$. Then, after some elementary calculus, we have $C_\lambda$ is I$(1,1,1)$, I$(1,-1,-1), I(-1,1,-1),$ and I$(-1,-1,1)$  for $\lambda\in[-1,0]$; and $C_\lambda$ is I$(-1,1,1)$, I$(1,-1,1)$, I$(1,1,-1)$ and I$(-1,-1,-1)$ for $\lambda\in[0,1]$.
\end{example}

\subsection{The multivariate case}

$\phantom{999}$Now, we will explore the $n$-dimensional case. To achieve our goals, we need the following preliminary lemma, which will allow us to relate the probability of the event $\bigcap_{i=1}^n (\alpha_i X_i > x_i)$ to the associated $n$-copula of the random vector $(X_1,X_2,\ldots,X_n)$ and its margins.

\begin{lemma}\label{lemma}
    Let $\XX=(X_1,X_2,\ldots,X_n)$ be an $n$-dimensional random vector and let $\alpha\in\R^n$ such that $|\alpha_i|=1$ for all $1\leq i\leq n$. Let $\alpha_k=-1$ for $k\in I=\{i_1,...,i_p\}\subseteq\{1,...,n\}$ and let $J=\{1,...,n\}/I$ $(I,J\neq\emptyset)$. For all $\xx \in \overline{\R}^n=[-\infty,\infty]^n$, we have 
    \begin{align*}
        \mathbb{P}\left[\bigcap_{i=1}^n(\alpha_iX_i>\overline{x}_i)\right]&=\mathbb{P}\left[\bigcap_{i\in I}(X_i\leq \overline{x}_i)\right]-\sum_{j\in J}\mathbb{P}\left[\bigcap_{i\in I\cup \{j\}}(X_i<\overline{x}_i)\right]\\
        &+\sum_{j_1\in J}\sum_{\substack{j_2\in J\\ j_2 >j_1}}\mathbb{P}\left[\bigcap_{i\in I\cup \{j_1,j_2\}}(X_i<\overline{x}_i)\right]-\cdots\\
        &+(-1)^{|J|-1}\sum_{j_1\in J}\sum_{\substack{j_2\in J\\ j_2 >j_1}}\cdots\sum_{\substack{j_{|J|-1}\in J\\ j_{|J|-1} >j_{|J|-2}}}\mathbb{P}\left[\bigcap_{i\in I\cup \{j_1,...,j_{|J|-1}\}}(X_i<\overline{x}_i)\right]\\
        &+(-1)^{|J|}\mathbb{P}\left[\bigcap_{i=1}^n(X_i<\overline{x}_i)\right],
    \end{align*}
    where 
    $$\overline{x}_i=\begin{cases}
        x_i, & \text{ if } i\in J,\\
        -x_i, & \text{ if } i\in I.
    \end{cases}$$
\end{lemma}

\begin{proof}
    We are applying the induction method over the number of components of $\alpha$ that are equal to $1$; i.e., over the cardinal of $J$. We will consider the following change of variable throughout this proof: if $-X_i>x_i$, then $X_i<-x_i$ so we consider $\overline{x}_i=-x_i$ if $i\in I$ and $\overline{x}_i=x_i$ if $i\in J$. Let us start proving the case $|J|=1$:
    If we suppose that $\alpha_l=1$ and the remaining coordinates are all equal to $-1$, then 
    \begin{align*}
        \mathbb{P}\left[\bigcap_{i=1}^n(\alpha_i X_i>x_i)\right]&=\mathbb{P}\left[x_l<X_l<\infty,\bigcap_{\substack{i=1\\ i\neq l}}^n(X_i<\overline{x}_i)\right]\\
        &=\mathbb{P}\left[\bigcap_{\substack{i=1\\ i\neq l}}^n(X_i<\overline{x}_i)\right]-\mathbb{P}\left[\bigcap_{i=1}^n(X_i<\overline{x}_i)\right].
    \end{align*}
    Now we consider the statement true for $|J|=m$ and our goal is claiming the same for $|J|=m+1$. So, if $l_{m+1}\in J$ we have 
    \begin{align*}
\mathbb{P}\left[\bigcap_{i=1}^n(\alpha_iX_i>x_i)\right]&=\mathbb{P}\left[\bigcap_{i\in J}(X_i>x_i),\bigcap_{i\in I}(X_i<x_i)\right]\\
        &=\mathbb{P}\left[x_{l_{m+1}}<X_{l_{m+1}}<\infty,\bigcap_{\substack{i\in J\\ i\neq l_{m+1}}}(X_i>x_i),\bigcap_{i\in I}(X_i<\overline{x}_i)\right]\\
        &=\mathbb{P}\left[\bigcap_{\substack{i\in J\\ i\neq l_{m+1}}}(X_i>x_i),\bigcap_{i\in I}(X_i<\overline{x}_i)\right]-\mathbb{P}\left[\bigcap_{\substack{i\in J\\ i\neq l_{m+1}}}(X_i>x_i),\bigcap_{i\in I\cup\{l_{m+1}\}}(X_i<\overline{x}_i)\right].
    \end{align*}
    Applying the induction hypothesis we obtain that the first expression is equal to
    \begin{align*}
        &\mathbb{P}\left[\bigcap_{\substack{i\in J\\ i\neq l_{m+1}}}(X_i>x_i),\bigcap_{i\in I}(X_i<\overline{x}_i)\right]=\mathbb{P}\left[\bigcap_{i\in I}(X_i<\overline{x}_i)\right]-\sum_{\substack{j\in J\\ j\neq l_{m+1}}}\mathbb{P}\left[\bigcap_{i\in I\cup\{j\}}(X_i<\overline{x}_i)\right]\\
        &+\sum_{\substack{j_1\in J\\ j_1\neq l_{m+1}}}\sum_{\substack{j_2\in J\\ j_2\neq l_{m+1}\\j_2>j_1}}\mathbb{P}\left[\bigcap_{i\in I\cup\{j_1,j_2\}}(X_i < \overline{x}_i)\right]-\cdots+(-1)^m\mathbb{P}\left[\bigcap_{\substack{i= 1\\ i\neq l_{m+1}}}^n(X_i<\overline{x}_i)\right].
    \end{align*}
    Whereas the second one is equal to the following expression
       \begin{align*}
        &\mathbb{P}\left[\bigcap_{\substack{i\in J\\ i\neq l_{m+1}}}(X_i>x_i),\bigcap_{i\in I\cup\{l_{m+1}\}}(X_i<\overline{x}_i)\right]=\mathbb{P}\left[\bigcap_{i\in I\cup\{l_{m+1}\}}(X_i<\overline{x}_i)\right]-\sum_{\substack{j\in J\\ j\neq l_{m+1}}}\mathbb{P}\left[\bigcap_{i\in I\cup\{l_{m+1},j\}}(X_i<\overline{x}_i)\right]\\
        &+\sum_{\substack{j_1\in J\\ j_1\neq l_{m+1}}}\sum_{\substack{j_2\in J\\ j_2\neq l_{m+1}\\j_2>j_1}}\mathbb{P}\left[\bigcap_{i\in I\cup\{l_{m+1},j_1,j_2\}}(X_i< \overline{x}_i)\right]-\cdots+(-1)^m\mathbb{P}\left[\bigcap_{i=1}^n(X_i<\overline{x}_i)\right].
    \end{align*}
    Now, we subtract both expressions. Note that we can add the first term of the second expression to the second term of the first one; ie., 
    $$-\sum_{\substack{j\in J\\ j\neq l_{m+1}}}\mathbb{P}\left[\bigcap_{i\in I\cup\{j\}}(X_i<\overline{x}_i)\right]-\mathbb{P}\left[\bigcap_{i\in I\cup\{l_{m+1}\}}(X_i<\overline{x}_i)\right]=-\sum_{j\in J}\mathbb{P}\left[\bigcap_{i\in I\cup\{j\}}(X_i<\overline{x}_i)\right].$$
    It follows similarly with the third term of the first expression and the second term of the second one
    \begin{align*}
        &\sum_{\substack{j_1\in J\\ j_1\neq l_{m+1}}}\sum_{\substack{j_2\in J\\ j_2\neq l_{m+1}\\j_2>j_1}}\mathbb{P}\left[\bigcap_{i\in I\cup\{j_1,j_2\}}(X_i< \overline{x}_i)\right]+\sum_{\substack{j\in J\\ j\neq l_{m+1}}}\mathbb{P}\left[\bigcap_{i\in I\cup\{l_{m+1},j\}}(X_i<\overline{x}_i)\right]\\
        &=\sum_{j_1\in J}\sum_{\substack{j_2\in J\\j_2>j_1}}\mathbb{P}\left[\bigcap_{i\in I\cup\{j_1,j_2\}}(X_i< \overline{x}_i)\right],
    \end{align*}
    and so on with the remaining terms. Finally, our first and last terms are, respectively, 
    $$\mathbb{P}\left[\bigcap_{i\in I}(X_i<\overline{x}_i)\right]\hspace{0.7cm}\text{and}\hspace{0.7cm}(-1)^{m+1}\mathbb{P}\left[\bigcap_{i=1}^n(X_i<\overline{x}_i)\right].$$
    
    Thus, grouping all of these terms together we reach the desired result.
\end{proof}

\begin{remark}\label{remark1}
    Note that if the $n$-dimensional random vector $\XX$ has associated  $n$-copula $C$, then the expression of Lemma \ref{lemma} can be written in terms of the copula and its margins via Sklar's theorem. First term, $\mathbb{P}\left[\bigcap_{i\in I}(X_i\leq x_i)\right]$, is equivalent to the  $p$-dimensional margin of the copula, $C_{i_1,...,i_p}(u)$, where $u_i=F_i(x_i)$ if $i\in I$ ($F_i$ is the $i$-margin of $\XX$) and $u_i=1$ otherwise. Second term, $\sum_{j\in J}\mathbb{P}\left[\bigcap_{i\in I\cup \{j\}}(X_i<x_i)\right]$, is equal to the sum of all $(p+1)$-dimensional margins of the copula $C$, involving all the subscripts in $I$ and every other subscript in $J$, i.e., $\sum_{j\in J}C_{j,i_1,...,i_p}(u)$, where, comparably, $u_i=F_i(x_i)$ if $i\in I\cup\{j\}$  and $u_i=1$ otherwise. So, expression of Lemma \ref{lemma} is equivalent to 
    \begin{align*}
        C_{i_1,...,i_p}(u)-\sum_{j\in J}C_{j,i_1,...,i_p}(u)+\sum_{j_1\in J}\sum_{\substack{j_2\in J\\ j_2>j_1}}C_{j_1,j_2,i_1,...,i_p}(u)-\cdots+(-1)^{|J|}C(u).
    \end{align*}
\end{remark}

Thanks to Remark \ref{remark1}, we are now able to state and prove the pertinent theorem for the multivariate I$(\alpha)$ dependence concept. We want to stress that in the following theorem, we will take the liberty of considering the subscripts that appear in the marginal copulas without regard to order to facilitate their writing; that is, for example, the marginal copula $C_{12}$ will be the same as the marginal copula $C_{21}$.  

\begin{theorem}\label{th:mult}
    Let $\UU=(U_1,U_2,\ldots,U_n)$ be an $n$-dimensional random vector with associated $n$-copula $C$. Let $\alpha\in\R^n$, with $|\alpha_i|=1$ for all $1\leq i \leq n$. Suppose $\alpha_l=-1$ for $l\in I=\{i_1,...,i_p\}\subseteq \{1,...,n\}$, and let $J$ be $\{1,...,n\}\backslash I$, with $I,J\neq\emptyset$. Then $C$ is $I(\alpha)$ if, and only if,
    \begin{align}\label{eq:mainthmult}\nonumber
&\left[\sum_{k=0}^{|J|}\left(\sum_{\substack{j_{k}=1\\ j_k\in J}}\sum_{\substack{j_{k-1}>j_k\\ j_{k-1}\in J}}\cdots\sum_{\substack{j_{k-(|J|-1)}>j_{k-|J|}\\ j_{k-(|J|-1)}\in J}}(-1)^kC_{i_1,...,i_p,j_k,...,j_{k-(|J|-1)}}(\overline{u})\right)\right]\cdot\\ \nonumber
        &\left[\sum_{k=0}^{|J|}\left(\sum_{\substack{j_{k}=1\\ j_k\in J}}\sum_{\substack{j_{k-1}>j_k\\ j_{k-1}\in J}}\cdots\sum_{\substack{j_{k-(|J|-1)}>j_{k-|J|}\\ j_{k-(|J|-1)}\in J}}(-1)^kC_{i_1,...,i_p,j_k,...,j_{k-(|J|-1)}}(\overline{u}')\right)\right]\leq\\\nonumber
        &\left[\sum_{k=0}^{|J|}\left(\sum_{\substack{j_{k}=1\\ j_k\in J}}\sum_{\substack{j_{k-1}>j_k\\ j_{k-1}\in J}}\cdots\sum_{\substack{j_{k-(|J|-1)}>j_{k-|J|}\\ j_{k-(|J|-1)}\in J}}(-1)^kC_{i_1,...,i_p,j_k,...,j_{k-(|J|-1)}}(\overline{u}_1,...,u_{i_1}',...,u_{i_p}',...,\overline{u}_n)\right)\right]\cdot\\
        &\left[\sum_{k=0}^{|J|}\left(\sum_{\substack{j_{k}=1\\ j_k\in J}}\sum_{\substack{j_{k-1}>j_k\\ j_{k-1}\in J}}\cdots\sum_{\substack{j_{k-(|J|-1)}>j_{k-|J|}\\ j_{k-(|J|-1)}\in J}}(-1)^kC_{i_1,...,i_p,j_k,...,j_{k-(|J|-1)}}(\overline{u}_1',...,u_{i_1},...,u_{i_p},...,\overline{u}_n')\right)\right],
    \end{align}
     for all $u,u'\in [0,1]^n$ such that $u\leq u'$, where
    $$\overline{u}_h=\begin{cases}
        u_h, & \text{if } \exists\hspace{0.1cm} l:  j_l=h \text{ or } i_l=h,\\
        1, & \text{otherwise}
    \end{cases} \hspace{0.2cm} \text{ and }\hspace{0.2cm} \overline{u}'_h=\begin{cases}
        u'_h, & \text{if } \exists\hspace{0.1cm} l:  j_l=h \text{ or } i_l=h,\\
        1, & \text{otherwise.}
    \end{cases} $$
\end{theorem}

\begin{remark}
    Before proving Theorem \ref{th:mult}, we must note that in the expression above if the subscripts $j_{r}\le 0$ we do not take them into account for summing. For example, if $k=1$ we only consider $j_k=j_1$ therefore the term in the sums is $-C_{j_1,i_1,...,i_p}$.
\end{remark}

\begin{proof} 
   (Theorem \ref{th:mult}) Let $\UU=(U_1,U_2,\ldots,U_n)$ be an $n$-dimensional random vector with associated $n$-copula $C$. Since $C$ is I$(\alpha)$, we have that
    $$\mathbb{P}\left[\bigcap_{i\in J}(U_i>u_i),\bigcap_{i\in I}(-U_i>u_i)|\bigcap_{i\in J}(U_i>u_i'),\bigcap_{i\in I}(-U_i>u_i')\right]$$
    is nondecreasing in ${\bf u'}$ for all ${\bf u}$, i.e., 
    \begin{align*}
         &\mathbb{P}\left[\bigcap_{i\in J}(U_i>u_i),\bigcap_{i\in I}(-U_i>u_i)|\bigcap_{i\in J}(U_i>u_i'),\bigcap_{i\in I}(-U_i>u_i')\right]\\
         \leq &\mathbb{P}\left[\bigcap_{i\in J}(U_i>u_i),\bigcap_{i\in I}(-U_i>u_i)|\bigcap_{i\in J}(U_i>u_i''),\bigcap_{i\in I}(-U_i>u_i'')\right], 
         \end{align*}
         for all $u_i,u_i',u_i''\in [0,1]$ such that $i\in J$, and $u_i,u_i',u_i''\in [-1,0]$ such that $i\in I$, verifying that $u_i'\le u_i''$ for all $1\leq i \leq n$. The inequality above is equivalent to the following expression
         \begin{align*}
             &\frac{\mathbb{P}\left[\bigcap_{i\in J}(U_i>u_i),\bigcap_{i\in I}(-U_i>u_i),\bigcap_{i\in J}(U_i>u_i'),\bigcap_{i\in I}(-U_i>u_i')\right]}{\mathbb{P}\left[\bigcap_{i\in J}(U_i>u_i'),\bigcap_{i\in I}(-U_i>u_i')\right]}\\
             \leq & \frac{\mathbb{P}\left[\bigcap_{i\in J}(U_i>u_i),\bigcap_{i\in I}(-U_i>u_i),\bigcap_{i\in J}(U_i>u_i''),\bigcap_{i\in I}(-U_i>u_i'')\right]}{\mathbb{P}\left[\bigcap_{i\in J}(U_i>u_i''),\bigcap_{i\in I}(-U_i>u_i'')\right]},
         \end{align*}
         i.e., 
        \begin{align*}
             &\frac{\mathbb{P}\left[\bigcap_{i\in J}(U_i>u_i),\bigcap_{i\in I}(U_i<-u_i),\bigcap_{i\in J}(U_i>u_i'),\bigcap_{i\in I}(U_i<-u_i')\right]}{\mathbb{P}\left[\bigcap_{i\in J}(U_i>u_i'),\bigcap_{i\in I}(U_i<-u_i')\right]}\\
             \leq\hspace{0.1cm} & \frac{\mathbb{P}\left[\bigcap_{i\in J}(U_i>u_i),\bigcap_{i\in I}(U_i<-u_i),\bigcap_{i\in J}(U_i>u_i''),\bigcap_{i\in I}(U_i<-u_i'')\right]}{\mathbb{P}\left[\bigcap_{i\in J}(U_i>u_i''),\bigcap_{i\in I}(U_i<-u_i'')\right]},
         \end{align*} 
         and therefore,
         \begin{align*}
             &\frac{\mathbb{P}\left[\bigcap_{i\in J}(U_i>\max\{u_i,u_i'\}),\bigcap_{i\in I}(U_i<\min\{-u_i,-u_i'\}),\right]}{\mathbb{P}\left[\bigcap_{i\in J}(U_i>u_i'),\bigcap_{i\in I}(U_i<-u_i')\right]}\\
             \leq\hspace{0.1cm} & \frac{\mathbb{P}\left[\bigcap_{i\in J}(U_i>\max\{u_i,u_i''\}),\bigcap_{i\in I}(U_i<\min\{-u_i,-u_i''\})\right]}{\mathbb{P}\left[\bigcap_{i\in J}(U_i>u_i''),\bigcap_{i\in I}(U_i<-u_i'')\right]}.
         \end{align*} 
     Using Lemma \ref{lemma} and Remark \ref{remark1}, it is possible to express this inequality in terms of copulas as follows  
     \begin{align*}
         &\frac{ C_{i_1,...,i_p}(u^{(1)})-\sum_{j\in J}C_{j,i_1,...,i_p}(u^{(1)})+\sum_{j_1\in J}\sum_{\substack{j_2\in J\\ j_2>j_1}}C_{j_1,j_2,i_1,...,i_p}(u^{(1)})-\cdots+(-1)^{|J|}C(u^{(1)})}{ C_{i_1,...,i_p}(\overline{u}')-\sum_{j\in J}C_{j,i_1,...,i_p}(\overline{u}')+\sum_{j_1\in J}\sum_{\substack{j_2\in J\\ j_2>j_1}}C_{j_1,j_2,i_1,...,i_p}(\overline{u}')-\cdots+(-1)^{|J|}C(\overline{u}')}\\
         \leq\hspace{0.1cm} & \frac{ C_{i_1,...,i_p}(u^{(2)})-\sum_{j\in J}C_{j,i_1,...,i_p}(u^{(2)})+\sum_{j_1\in J}\sum_{\substack{j_2\in J\\ j_2>j_1}}C_{j_1,j_2,i_1,...,i_p}(u^{(2)})-\cdots+(-1)^{|J|}C(u^{(2)})}{ C_{i_1,...,i_p}\overline{u}'')-\sum_{j\in J}C_{j,i_1,...,i_p}(\overline{u}'')+\sum_{j_1\in J}\sum_{\substack{j_2\in J\\ j_2>j_1}}C_{j_1,j_2,i_1,...,i_p}(\overline{u}'')-\cdots+(-1)^{|J|}C(\overline{u}'')},
     \end{align*}
    where 
    $$u^{(1)}_i=\begin{cases}
        \max\{u_i,u_i'\}, & \text{ if } i\in J;\\
        \min\{-u_i,-u_i'\}, & \text{ if } i\in I;
    \end{cases} \hspace{0.5cm}\text{and}\hspace{0.5cm} u^{(2)}_i=\begin{cases}
        \max\{u_i,u_i''\}, & \text{ if } i\in J;\\
        \min\{-u_i,-u_i''\}, & \text{ if } i\in I,
    \end{cases}$$
and
       $$\overline{u}'_i=\begin{cases}
        u'_i, & \text{ if } i\in J;\\
        -u'_i, & \text{ if } i\in I;
    \end{cases} \hspace{0.5cm}\text{and}\hspace{0.5cm}\overline{u}''_i=\begin{cases}
       u''_i , & \text{ if } i\in J;\\
        -u''_i, & \text{ if } i\in I.
    \end{cases}$$
    
     By setting $u_i=u_i'\leq u_i''$ if $i\in J$ and $u_i\leq u_i'=u_i''$ if $i \in I$, we obtain 
     \begin{align*}
         &\left[C_{i_1,...,i_p}(u)-\sum_{j\in J}C_{j,i_1,...,i_p}(u)+\sum_{j_1\in J}\sum_{\substack{j_2\in J\\ j_2>j_1}}C_{j_1,j_2,i_1,...,i_p}(u)-\cdots+(-1)^{|J|}C(u)\right]\cdot\\
         &\left[C_{i_1,...,i_p}(u')-\sum_{j\in J}C_{j,i_1,...,i_p}(u')+\sum_{j_1\in J}\sum_{\substack{j_2\in J\\ j_2>j_1}}C_{j_1,j_2,i_1,...,i_p}(u')-\cdots+(-1)^{|J|}C(u')\right]\\
         &\leq \left[C_{i_1,...,i_p}(u_1,...,u_{i_1}',...,u_{i_p}',...,u_n)-\sum_{j\in J}C_{j,i_1,...,i_p}(u_1,...,u_{i_1}',...,u_{i_p}',...,u_n)\right.\\
         &\left.+\sum_{j_1\in J}\sum_{\substack{j_2\in J\\ j_2>j_1}}C_{j_1,j_2,i_1,...,i_p}((u_1,...,u_{i_1}',...,u_{i_p}',...,u_n)-\cdots+(-1)^{|J|}C(u_1,...,u_{i_1}',...,u_{i_p}',...,u_n)\right]\cdot\\
         &\left[C_{i_1,...,i_p}(u_1',...,u_{i_1},...,u_{i_p},...,u'_n)-\sum_{j\in J}C_{j,i_1,...,i_p}(u_1',...,u_{i_1},...,u_{i_p},...,u_n')\right.\\
         &\left. +\sum_{j_1\in J}\sum_{\substack{j_2\in J\\ j_2>j_1}}C_{j_1,j_2,i_1,...,i_p}(u_1',...,u_{i_1},...,u_{i_p},...,u_n')-\cdots+(-1)^{|J|}C(u_1',...,u_{i_1},...,u_{i_p},...,u_n')\right].
     \end{align*}
It is possible to compact the expression above as it follows
$$\sum_{k=0}^{|J|}\left(\sum_{\substack{j_{k}=1\\ j_k\in J}}\sum_{\substack{j_{k-1}>j_k\\ j_{k-1}\in J}}\cdots\sum_{\substack{j_{k-(|J|-1)}>j_{k-|J|}\\ j_{k-(|J|-1)}\in J}}(-1)^kC_{i_1,...,i_p,j_k,...,j_{k-(|J|-1)}}(u)\right).$$
Thus, we obtain the desired inequality.

Conversely, if we suppose that the Inequality \eqref{eq:mainthmult} holds for every $\uu,\uu'\in [0,1]^n$ such that $u_i\leq u_i'$ for all $1\leq i\leq n$, given $\uu,\uu'\in [0,1]^n$ such that $u_i, u_i'\in [0,1]$ if $i\in J$ and $u_i,u_i'\in[-1,0]$ if $i\in I$ we consider three possible cases in order to prove that $C$ is I$(\alpha)$:
     \begin{itemize}
         \item [i)] If $u_i> u_i'$ for all $1\le i\le n$ we have that
         $$\mathbb{P}\left[\bigcap_{i\in J}(U_i>u_i),\bigcap_{i\in I}(-U_i>u_i)|\bigcap_{i\in J}(U_i>u_i'),\bigcap_{i\in I}(-U_i>u_i')\right]=\frac{\mathbb{P}\left[\bigcap_{i\in J}(U_i>u_i),\bigcap_{i\in I}(-U_i>u_i)\right]}{\mathbb{P}\left[\bigcap_{i\in J}(U_i>u_i'),\bigcap_{i\in I}(-U_i>u_i')\right]}$$
         is nondecreasing in ${\bf u'}$.
         \item [ii)] If $u_i\leq u_i'$ for all $1\le i\le n$ we have
         $$\mathbb{P}\left[\bigcap_{i\in J}(U_i>u_i),\bigcap_{i\in I}(-U_i>u_i)|\bigcap_{i\in J}(U_i>u_i'),\bigcap_{i\in I}(-U_i>u_i')\right]=1,$$
         and then it is nondecreasing in ${\bf u'}$.
         \item[iii)] Let us consider, without loss of generality, $u_i\leq u_i'$ for $i\in J$ and $u_i> u_i'$ for $i\in I$. Then we have 
         \begin{align}\label{eq:case3}\nonumber
             \mathbb{P}[\alpha {\bf U}> {\bf u}| \alpha {\bf U}>{\bf u'}]&=\frac{\mathbb{P}\left[\bigcap_{i=1}^n(\alpha_iU_i>u_i),\bigcap_{i=1}^n(\alpha_iU_i>u_i')\right]}{\mathbb{P}\left[\bigcap_{i=1}^n(\alpha_i U_i>u_i')\right]}\\
             &=\frac{\mathbb{P}\left[\bigcap_{i\in J}(U_i>u_i'),\bigcap_{i\in I}(-U_i>u_i)\right]}{\mathbb{P}\left[\bigcap_{i=1}^n(\alpha_i U_i>u_i')\right]}
         \end{align}
         Since $\mathbb{P}\left[\bigcap_{i=1}^n(\alpha_i U_i>u_i')\right]\geq \mathbb{P}\left[\bigcap_{i=1}^n(\alpha_iU_i>u_i'')\right]$ for all $u_i',u_i''$ such that $u_i'\leq u_i''$ for $1\leq i\leq n$, we obtain that Expression \eqref{eq:case3} is nondecreasing in $u_i'$ for $i\in I$. To prove that it is also nondrecreasing in $u_i'$ for $i\in J$ considering $u_i''$ such that $u_i'\leq u_i''$ for $i\in J$, we need to verify that
         $$\mathbb{P}\left[\bigcap_{i\in I}(-U_i>u_i)|\bigcap_{i=1}^n(\alpha_iU_i>u_i')\right]\leq\mathbb{P}\left[\bigcap_{i\in I}(-U_i>u_i)|\bigcap_{i\in J}(U_i>u_i''),\bigcap_{i\in I}(-U_i>u_i')\right].$$
         This inequality is equivalent to 
             $$\frac{\mathbb{P}\left[\bigcap_{i\in J}(U_i>u_i'),\bigcap_{i\in I}(-U_i>u_i)\right]}{\mathbb{P}\left[\bigcap_{i=1}^n(\alpha_iU_i>u_i')\right]}\leq\frac{\mathbb{P}\left[\bigcap_{i\in J}(U_i>u_i''),\bigcap_{i\in I}(-U_i>u_i)\right]}{\mathbb{P}\left[\bigcap_{i\in J}(U_i>u_i''),\bigcap_{i\in I}(-U_i>u_i')\right]}.$$
        Rewriting this last inequality in terms of copulas, and denoting
        \begin{align*}
        &u^{(1)}=(u_1',\ldots,u_{i_1},\ldots,u_{i_p},\ldots,u_n'),\\
        &u^{(2)}=(u_1',\ldots,u'_{i_1},\ldots,u'_{i_p},\ldots,u_n'),\\
        &u^{(3)}=(u_1'',\ldots,u_{i_1},\ldots,u_{i_p},\ldots,u_n''),\\
        &u^{(4)}=(u_1'',\ldots,u'_{i_1},\ldots,u'_{i_p},\ldots,u_n''),
        \end{align*}we obtain
        \begin{align*}
            &\frac{C_{i_1,\ldots,i_p}(u^{(1)})-\sum_{j\in J}C_{j,i_1,...,i_p}(u^{(1)}))+\sum_{j_1\in J}\sum_{\substack{j_2\in J\\ j_2>j_1}}C_{j_1,j_2,i_1,...,i_p}(u^{(1)})-\cdots+(-1)^{|J|}C(u^{(1)})}{ C_{i_1,...,i_p}(u^{(2)})-\sum_{j\in J}C_{j,i_1,...,i_p}(u^{(2)})+\sum_{j_1\in J}\sum_{\substack{j_2\in J\\ j_2>j_1}}C_{j_1,j_2,i_1,...,i_p}(u^{(2)})-\cdots+(-1)^{|J|}C(u^{(2)})}\\
            \leq & \frac{C_{i_1,...,i_p}(u^{(3)})-\sum_{j\in J}C_{j,i_1,...,i_p}(u^{(3)}))+\sum_{j_1\in J}\sum_{\substack{j_2\in J\\ j_2>j_1}}C_{j_1,j_2,i_1,...,i_p}(u^{(3)})-\cdots+(-1)^{|J|}C(u^{(3)})}{ C_{i_1,...,i_p}(u^{(4)})-\sum_{j\in J}C_{j,i_1,...,i_p}(u^{(4)})+\sum_{j_1\in J}\sum_{\substack{j_2\in J\\ j_2>j_1}}C_{j_1,j_2,i_1,...,i_p}(u^{(4)})-\cdots+(-1)^{|J|}C(u^{(4)})}.
        \end{align*}
        This inequality holds since it is equivalent to Inequality \eqref{eq:mainthmult}.
     \end{itemize}
Therefore, the proof is complete.
\end{proof}

Next, we will provide several examples to illustrate the I$(\alpha)$ concept within the context of the multivariate case.

\begin{example}\label{ex6}
    For all $n\ge 2$, the $n$-copula $M^n$ is I$(\bf 1)$ and I$(-\bf 1)$.
\end{example}

\begin{example}\label{ex7}
For all $n\ge 2$, the $n$-copula $\Pi^n$ is I$(\alpha)$ for any direction $\alpha$.
\end{example}

\begin{example}
    Let $C$ be the $n$-copula given by the convex linear combination of the $n$-copulas $\Pi^n$ and $M^n$, i.e.,
    $$C(\uu)=\theta\Pi^n(\uu)+(1-\theta)M^n(\uu)$$
    for all $\uu\in[0,1]^n$, where $\theta\in[0,1]$. As a consequence of Examples \ref{ex6} and \ref{ex7}, the $n$-copula $C$ is I$({\bf 1})$ and I$(-\bf 1)$ for all $\theta\in[0,1]$.
\end{example}

\begin{example}
    Let $\{C_{\lambda}\}_{\lambda \in [-1,1]}$ be the Farlie-Gumbel-Morgenstern $(FGM)$ one-parameter family of $n$-copulas given by
$$C_{\lambda}(\bf u)=\prod_{i=1}^nu_i\left[1+\lambda\prod_{i=1}^n(1-u_i)\right]$$
for all $\bf u\in[0,1]^n$. Then we have that if $|J|$ is even, $C_\lambda$ is I$(\alpha)$ for $\lambda\in[0,1]$, where $J=\{i\in\{1,...,n\}:\alpha_i=1\}$; and  if $|J|$ is odd, $C_\lambda$ is I$(\alpha)$ for $\lambda\in[-1,0]$.
\end{example}

\section{Conclusions}\label{sec:4}

$\phantom{999}$In this paper we have characterized the monotonicity according to a direction, a new concept of positive dependence, ---particularly the $I(\alpha)$ notion--- by copulas. We started with the bivariate and trivariate cases, due to their simplicity and straightforwardness, aiming to illustrate the procedure to be followed in the multivariate case. Initially, we obtained in both cases a characterization in terms of an inequality involving the associated copula and its margins evaluated at maximum and minimum values, in order to subsequently simplify this characterization for application to specific cases. Throughout this paper, examples of copulas with $I(\alpha)$ dependence have been provided for different values of $\alpha$.

%Some topics of interest for future research related %to monotonicity according to a direction and copulas %could be obtaining a characterization of this %dependence property based on the generator of %Archimedean copulas...

\end{document}